\documentclass{amsart}

\usepackage{graphicx}
\usepackage{amsmath,amssymb}
\usepackage{stmaryrd}

\graphicspath{{Figures/}} 
\graphicspath{{./Figures/}}

\usepackage[left=4cm, right=4cm, top=4cm, bottom=4cm]{geometry}

\def\cA{\mathcal{A}}
\def\cB{\mathcal{B}}
\def\cG{\mathcal{G}}
\def\cI{\mathcal{I}}
\def\cS{\mathcal{S}}
\def\cT{\mathcal{T}}
\def\cU{\mathcal{U}}
\def\cN{\mathcal{N}}
\def\cP{\mathcal{P}}
\def\cQ{\mathcal{Q}}
\def\cR{\mathcal{R}}
\def\cW{\mathcal{W}}
\def\push{\mathrm{push}}
\def\can{\mathrm{can}}
\def\Tam{\mathrm{Tam}}

\def\nblue{n_{\mathrm{blue}}}
\def\nblack{n_{\mathrm{black}}}
\def\nred{n_{\mathrm{red}}}

\def\Sep{\mathrm{Sep}}
\def\Tblue{T_{\mathrm{blue}}}

\def\low{\mathrm{low}}
\def\mid{\mathrm{mid}}
\def\upp{\mathrm{up}}

\def\gammal{\gamma_\low}
\def\gammam{\gamma_\mid}
\def\gammau{\gamma_\upp}

\newtheorem{theo}{Theorem}
\newtheorem{prop}{Proposition}
\newtheorem{corollary}{Corollary}
\newtheorem{lemma}{Lemma}

\theoremstyle{remark}

\newtheorem{remark}{Remark}

\begin{document}


\title[Bijections for generalized Tamari intervals via orientations]{Bijections for generalized Tamari intervals via orientations}


\author{\'Eric Fusy}
\address{\'EF: Laboratoire Informatique Gaspard Monge, 
Universit\'e Gustave Eiffel, CNRS, LIGM, F-77454 Marne-la-Vallée, France}
\curraddr{}
\email{eric.fusy@u-pem.fr}
\urladdr{}


\author{Abel Humbert}
\address{AH: IRIF, Universit\'e Paris-Diderot, France}
\email{ahumbert@irif.fr}







\begin{abstract}
Generalized Tamari intervals have been recently introduced by Pr\'eville-Ratelle and Viennot, and have been proved to be in bijection with (rooted planar) non-separable maps by Fang and Pr\'eville-Ratelle. We present two new bijections between generalized Tamari intervals and non-separable maps.  Our first construction 
proceeds via separating decompositions on simple bipartite quadrangulations (which are known to be in bijection with non-separable maps). It can be seen as an extension of the Bernardi-Bonichon bijection between Tamari intervals and minimal Schnyder woods.  
On the other hand, our second construction relies on a specialization of the Bernardi-Bonichon bijection to so-called 
synchronized Tamari intervals, which are known to be in one-to-one correspondence with generalized Tamari intervals. 
It yields a trivariate generating function expression that interpolates between the bivariate generating function for generalized Tamari intervals, 
and the univariate generating function for Tamari intervals. 
\end{abstract}

\maketitle



\section{Introduction}
The $\nu$-Tamari lattice $\Tam(\nu)$ (for $\nu$ an arbitrary directed walk with steps in $\{N,E\}$) has been recently 
introduced by Pr\'eville-Ratelle and Viennot~\cite{PV17}, and further studied in~\cite{ceballos2019geometry,ceballos2018nu}, with connections to geometric combinatorics.  
It is a lattice on the set of directed walks weakly above $\nu$ and with same endpoints as $\nu$, and it 
generalizes the Tamari lattice~\cite{tamari1951monoi} (in size $n$, case where $\nu=(NE)^n$) 
and the $m$-Tamari lattices~\cite{bergeron2012higher} 
(in size $n$, case where $\nu=(NE^m)^n$).   
 
The enumeration of intervals (i.e., pairs formed by two elements $x,x'$ with $x\leq x'$) in Tamari lattices has attracted a lot of attention~\cite{bousquet2013representation,BFP12,Ch06}, due in particular to their (conjectural) connections to dimensions of diagonal coinvariant spaces~\cite{bergeron2012higher}, and to their bijective connections to planar maps~\cite{BB09}, as well as intriguing symmetry properties~\cite{Ch18,pons2018rise}. A \emph{planar map} (shortly, a map) is a connected multigraph embedded in the plane, up to continuous deformation. A \emph{rooted map} is a map with a marked corner incident to the outer face (all maps in this article are assumed to be rooted if not specified otherwise). A map is called \emph{non-separable} (or 2-connected) if it is either the loop-map, or is loopless and $M\backslash v$ is connected for every vertex $v\in M$. Chapoton~\cite{Ch06} proved that the number of Tamari intervals of size $n$ is $\frac{2}{n(n+1)}\binom{4n+1}{n-1}$, which coincides with the number of simple triangulations with $n+3$ vertices~\cite{Tu62}. Bernardi and Bonichon~\cite{BB09} subsequently gave a bijective proof of this formula, relying on so-called \emph{Schnyder woods} (orientations and colorations of the inner edges in $3$ colors, with specific local constraints).     
Regarding $\nu$-Tamari lattices, if we let $\cI_{\nu}$ be the set of intervals in $\Tam(\nu)$, then it has recently been shown by Fang and Pr\'eville-Ratelle~\cite{FP17} that 
$\cG_n:=\cup_{\nu\in\{N,E\}^n}\cI_{\nu}$ (generalized Tamari intervals of size $n$) is in bijection with the set $\cN_n$ of non-separable maps with $n+2$ edges, and more precisely that $\cG_{i,j}:=\sum_{\nu\in\mathfrak{S}(E^iN^j)}\cI_{\nu}$ is in bijection with the set $\cN_{i,j}$ of non-separable maps with $i+2$ vertices and $j+2$ faces (it is known~\cite{BT64,Tu63} that $|\cN_{n}|=\frac{2(3n+3)!}{(n+2)!(2n+3)!}$ and $|\cN_{i,j}|=\frac{(2i+j+1)!(2j+i+1)!}{(i+1)!(j+1)!(2i+1)!(2j+1)!}$). 
They have a first recursive bijection based on parallel decompositions with a catalytic variable, and then make the bijection more explicit via certain auxiliary labeled trees. 
As shown in~\cite{fang2018trinity}, their bijection also has interesting symmetry properties, as it commutes with natural involutions on the two classes (duality on maps, and a mirror duality for Tamari intervals).    

A \emph{quadrangulation} is a map with all faces of degree $4$; by a \emph{bipartite quadrangulation} we mean a quadrangulation endowed with its unique coloration of 
vertices in black or white such that adjacent vertices have different colors,   
and the root-vertex (vertex at the root-corner) is black. By a classical correspondence~\cite[Section 7]{B65}, $\cN_{i,j}$
 is in bijection with the set $\cQ_{i,j}$ of bipartite simple quadrangulations with $i+2$
black vertices and $j+2$ white vertices.

In this article, we give two new bijections between $\cG_{i,j}$ and $\cQ_{i,j}$. Each one relies on seeing $\cG_{i,j}$ as included in a certain superfamily, and specializing a bijection
involving oriented maps. In our first bijection (Section~\ref{sec:bij_sep}) 
we see $\cG_{i,j}$ as a subfamily of non-intersecting
triples of lattice walks (a so-called Baxter family) and specialize a bijection (closely related to the one in~\cite{FPS09}
and also to a recent bijection by Kenyon et al.~\cite{KMSW}) 
with so-called separating decompositions on simple quadrangulations. We also show that this construction gives an extension of the Bernardi-Bonichon bijection (which is recovered as the case $\nu=(NE)^n$).   
In our second bijection (Section~\ref{sec:bij_sch}) we see $\cG_{i,j}$ as a subfamily (synchronized intervals) 
of classical Tamari intervals of size $i+j+1$, to which we specialize the Bernardi-Bonichon bijection~\cite{BB09}, which we compose with a bijection~\cite{BF12} to certain tree-structures  on which we can characterize the property of being synchronized. 

Several parameters can be tracked by the first construction, which 
 gives a model of maps for intervals in the $m$-Tamari lattices, and reveals certain symmetry properties on $\cG_{i,j}$. The second construction yields a trivariate generating function expression (Corollary~\ref{coro:tri_series}) 
that interpolates between the bivariate generating function of generalized Tamari intervals and the univariate 
generating function of classical Tamari intervals. 

\section{The $\nu$-Tamari lattice, and generalized Tamari intervals}
We recall from~\cite{PV17} the definition of $\nu$-Tamari lattices, and how they are related to the classical Tamari lattice. 
We consider walks in $\mathbb{N}^2$ starting at the origin and having steps North or East 
(these are equivalent to words on the alphabet $\{N,E\}$). For two such walks $\gamma,\gamma'$, we say that $\gamma'$ 
is \emph{weakly above} $\gamma$ if $\gamma$ and $\gamma'$ have the same endpoint, and no East step of $\gamma$ is strictly
above the East step of $\gamma'$ in the same vertical column. A Dyck walk of length $2n$ thus corresponds to a walk $\gamma$ that is weakly above $(NE)^n$. More generally, for $\nu$ a walk, we let $\cW_{\nu}$ be the set of walks weakly above $\nu$. 
For $\gamma\in\cW_{\nu}$ and for $p=(x,y)$ a point on $\gamma$, we let $x'\geq x$ be the abscissa of the North step of $\nu$ from ordinate $y$ to $y+1$ (with the convention that $x'=i$ if $y=j$), and we let $\ell(p):=x'-x$. If $p$ is preceded by $E$ and followed by $N$ we let $p'$ be the next point after $p$ along $\gamma$ such that $\ell(p')=\ell(p)$, and we let $\push_p(\gamma)$ be the walk $\gamma'$ obtained from $\gamma$ by moving the $E$ preceding $p$ to be just after $p'$ (see Figure~\ref{fig:flip} for an example); we say that $\gamma'$ \emph{covers} $\gamma$. The \emph{Tamari lattice for $\nu$} is defined as  $\Tam(\nu)=(\cW_{\nu},\leq)$ where $\leq$ is the 
transitive closure of the covering relation.  
The classical Tamari lattice $\Tam_n$ corresponds to the special case $\Tam_n=\Tam((NE)^n)$, and more generally for $m\geq 1$, the $m$-Tamari lattice $\Tam_n^{(m)}$ corresponds 
to the special case $\Tam_n^{(m)}=\Tam((NE^m)^n)$ . 

\begin{figure}[h!]
\begin{center}
\includegraphics[width=10cm]{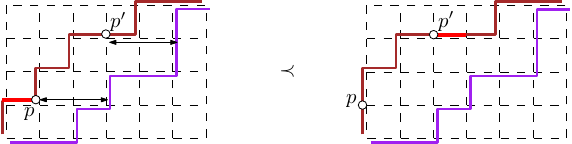}
\end{center}
\caption{A covering relation in $\Tam(\nu)$ for $\nu=EENENEENNE$.}
\label{fig:flip}
\end{figure}

Interestingly, for $\nu$ of length $n$, $\Tam(\nu)$ can also be obtained
 as a sublattice of $\Tam_{n+1}$, the classical Tamari lattice on Dyck walks of length $2n+2$. 
For $\gamma=E^{\alpha_0}NE^{\alpha_1}\cdots NE^{\alpha_k}$ a Dyck walk of length $2k$, 
the \emph{canopy-word} of $\gamma$ is the word $\can(\gamma)=(w_0,\ldots,w_{k})\in\{E,N\}^{k+1}$ such that for 
$r\in\llbracket 0,k\rrbracket$, $w_r=E$ if $\alpha_r=0$ and $w_r=N$ if $\alpha_r\geq 1$  
 (note that we always have $w_0=E$ and $w_k=N$). 
 Then $\Tam(\nu)$ is isomorphic to the sublattice of $\Tam_{n+1}$ induced by the Dyck walks
whose canopy-word is equal to $E\nu\ \!\!N$.


Let $\cG_{i,j}$ (resp. $\cG_n$) be the set of triples $(\nu,\gamma,\gamma')$ such 
that $\gamma\leq\gamma'$ in $\Tam(\nu)$, and $\nu$ ends at $(i,j)$ (resp. $\nu$ has length $n$).  
Elements of $\cG_{i,j}$ (resp. $\cG_n$) are called \emph{generalized Tamari intervals}
 with endpoint $(i,j)$ (resp. of size $n$).  
We now make two remarks based on properties shown in~\cite{PV17} (each remark is associated with a bijection for $\cG_{i,j}$ described later, respectively in Section~\ref{sec:bij_sep} and Section~\ref{sec:bij_sch}, the second remark is also used for the bijection in~\cite{FP17}). 

\begin{remark}
 Since $\gamma\leq \gamma'$ in $\Tam(\nu)$ implies that $\gamma'$ is weakly above $\gamma$, $\cG_{i,j}$
is a subfamily of the family $\cR_{i,j}$ of triples of walks $(\nu,\gamma,\gamma')$, each starting at the origin and ending at $(i,j)$, such that
 $\gamma'$ is weakly above $\gamma$, itself weakly above~$\nu$. 
\end{remark}

\begin{remark}
On the other hand, let $\cI_n$ be the set of intervals in $\Tam_n$ (classical Tamari intervals, on Dyck words of length $2n$). 
An interval $(\gamma,\gamma')\in\cI_n$ is called \emph{synchronized} if $\can(\gamma)=\can(\gamma')$. Let $\cS_{n}\subset\cI_n$ 
be the set of synchronized Tamari intervals of size $n$. 
Then the above sublattice characterization of $\Tam(\nu)$ implies that 
$\cG_n$ is in bijection with $\cS_{n+1}$. More generally, if we let $\cS_{i,j}$ be the set of synchronized intervals such that the common canopy-word is in $\frak{S}(E^{i+1}N^{j+1})$, then $\cG_{i,j}$ is in bijection with $\cS_{i,j}$.  
\end{remark}

\section{Bijection using separating decompositions}\label{sec:bij_sep}
Several bijections are known between $\cR_{i,j}$ and other combinatorial families (which are called \emph{Baxter families}, a survey is given in~\cite{FFNO11}). 
Our aim here is to pick one such bijection and show that it specializes 
nicely to the subfamily $\cG_{i,j}\subset\cR_{i,j}$. We pick the bijection (called here $\Phi$) from~\cite{FPS09} for separating decompositions,  but have to slightly 
modify it (the modified bijection is called $\Phi'$) so that it specializes well. As we will see in Section~\ref{sec:bijKMSW}, our construction $\Phi'$ is also closely related to a recent bijection by Kenyon et al.~\cite{KMSW}.   

\subsection{Separating decompositions}
For $Q\in\cQ_{i,j}$, let $s,s',t,t'$ be the outer vertices of $Q$ 
in clockwise order around the outer face, with $s$ the one
at the root.  
A \emph{separating decomposition} of $Q$ is given by an orientation and coloration (blue or red) of each edge of $Q$ such that:
\begin{itemize}
 \item
 All edges incident to $s$ (resp. $t$) are incoming blue (resp. incoming red).
 \item
  Every vertex $v\notin\{s,t\}$ has one outgoing edge in each color. Moreover, if $v$ is white (resp. black), then every incoming edge at $v$
   has the color of the next outgoing edge in clockwise (resp. counterclockwise) order around $v$, see Figure~\ref{fig:separating}(a). 
 \end{itemize} 
  
  An example of separating
decomposition is given in Figure~\ref{fig:separating}(b).  
It can be shown~\cite{FOP95} that the blue edges form a spanning tree of $Q\backslash t$ and the red 
edges form a spanning tree of $Q\backslash s$. By a slight abuse of notation, we also call \emph{separating decomposition} a pair $S=(Q,X)$,
where $Q$ is a simple quadrangulation, and $X$ is a separating decomposition on $Q$. 
We let $\Sep_{i,j}$ be the set of separating decompositions with $i+2$ black vertices and $j+2$ white vertices. 
A separating decomposition is called \emph{minimal}  if it has no clockwise cycle. 

\begin{figure}[h!]
\begin{center}
\includegraphics[width=12cm]{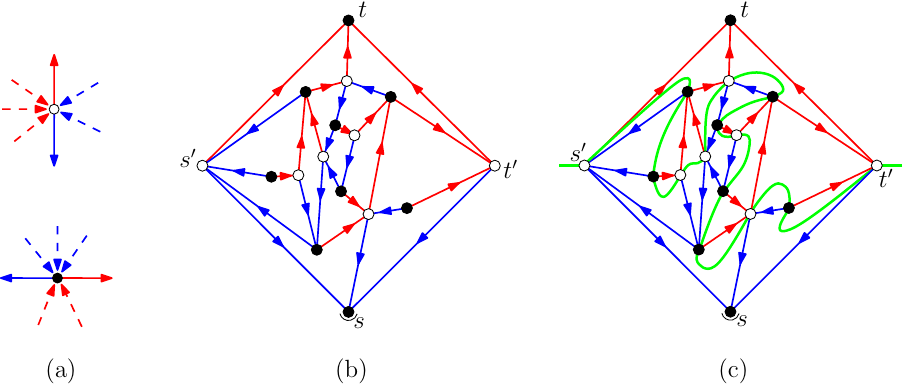}
\end{center}
\caption{(a) Local rule of separating decompositions for vertices not in $\{s,t\}$. (b) A separating 
decomposition. (c) The equatorial line (in green).}
\label{fig:separating}
\end{figure}

\begin{figure}[h!]
\begin{center}
\includegraphics[width=12cm]{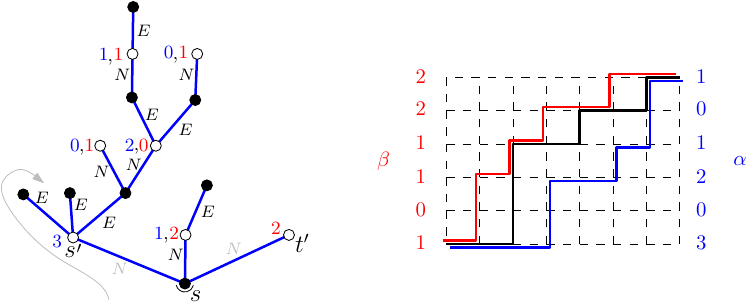}
\end{center}
\caption{Left: the blue tree of the separating decomposition of Figure~\ref{fig:separating}(b), with the indication of blue and red indegrees at white vertices. Right: the corresponding (by $\Phi'$) triple of walks.}
\label{fig:bijection}
\end{figure}

A general property of outdegree-constrained orientations of planar maps~\cite{Fe03}  
ensures that each simple quadrangulation
 has a unique minimal  separating decomposition. Hence, $\cQ_{i,j}$ is in bijective correspondence to the set of separating decompositions in $\Sep_{i,j}$ that are minimal.   

\subsection{Presentation and statement of the bijection}
We first recall the bijection $\Phi$ introduced in~\cite{FPS09} 
 between $\Sep_{i,j}$ and $\cR_{i,j}$. For  $S\in\Sep_{i,j}$, we let 
$\Tblue$ be the blue tree of $S$, and let $v_0,\ldots,v_{j+1}$ be the  
white vertices, ordered according to first visit in a clockwise walk around $\Tblue$ starting
at the root. For $r\in\llbracket 0,j\rrbracket$, let $\beta_r$ be the number of incoming red edges at $v_{r+1}$. 
Then $\Phi(S)$ is the triple of walks $(\gamma_{\low},\gamma_{\mid},\gamma_{\upp})$ (written here as binary words) obtained as follows: 
\begin{itemize}
\item
Let $w_{1}$ be the word obtained from a clockwise walk around $\Tblue$, where we write an $E$ each time we traverse an edge from white to black while getting farther from the root,  and write an $N$ each time we traverse an edge from white to black while getting closer to the root. 
Since the rightmost child $t'$ of $s$ is a leaf of $\Tblue$, $w_1$ ends with two occurences of $N$.
Let $\gamma_{\low}$ be $w_1$ without its two last $N$ letters,
\item
Let $w_2$ be the word obtained from a clockwise walk around $\Tblue$, where we write an $N$ each time we traverse an edge from black to white while getting farther from the root,  and write an $E$ each time we traverse an edge from black to white while getting closer to the root. Then $w_2$ starts with $N$, and (again due to the rightmost child of $s$ being a leaf) ends with $N$. 
Let $\gamma_{\mid}$ be $w_2$ without its first and last $N$ letters.
\item 
The walk $\gamma_{\upp}$ is $E^{\beta_0}NE^{\beta_1}\ldots NE^{\beta_{j}}$.
\end{itemize}

We now introduce a mapping $\Phi'$ that is a modified version of $\Phi$ (see Figure~\ref{fig:bijection} for an example), better suited in view
of the specialization to generalized Tamari intervals.   
 For $S$ a separating decomposition, 
 $\Phi'(S)$ is the triple  $(\gamma_{\low},\gamma_{\mid},\gamma_{\upp})$ of walks 
where $\gamma_{\mid}$ and $\gamma_{\upp}$ are obtained as above, and $\gamma_{\low}$ is modified to be 
$E^{\alpha_0}NE^{\alpha_1}\ldots NE^{\alpha_{j}}$, with $\alpha_r$ the number of incoming blue edges at $v_r$ for $r\in\llbracket 0,j\rrbracket$. 
For instance, for the separating decomposition of Figure~\ref{fig:bijection}, $\gamma_{\low}$ is $EEENEENENNNE$ when applying $\Phi$, and is $EEENNEENENNE$ when applying $\Phi'$. 
 
\begin{theo}\label{theo:1}
For $i,j\geq 0$, the mapping $\Phi'$ is a bijection between $\Sep_{i,j}$ and $\cR_{i,j}$. In addition,
for $S\in\Sep_{i,j}$, $S$ is minimal if and only if $\Phi'(S)\in\cG_{i,j}$. Hence, $\Phi'$ yields a bijection
between $\cQ_{i,j}$ and $\cG_{i,j}$.  
\end{theo}

The proof is delayed to Section~\ref{sec:proof}. 

\begin{remark}\label{rk:further}
 For $(\gammal,\gammam,\gammau)\in\cR_{i,j}$,  
a value $r\in\llbracket 1,j\rrbracket$ is called a \emph{level-value of type $(p,q)$}  
if $\beta_{r-1}=p$ and $\alpha_r=q$. 
On the other hand, for $Q\in\Sep_{i,j}$, an inner white vertex is said to be \emph{of type $(p,q)$} 
if it has $p$ incoming red edges and $q$ incoming blue edges.  
Then clearly in the bijection 
$\Phi'$, $\alpha_0$ is mapped to the degree of $s'$ minus $2$, $\beta_j$ is mapped to the degree of $t'$ minus $2$, and 
each level-value in $\llbracket 1,j\rrbracket$ corresponds to an inner white vertex of the same type.
\end{remark}
 

\begin{remark}
From the parameter-correspondence in Remark~\ref{rk:further}, we can see that our bijection for $\cG_{i,j}$ differs  (under the classical
correspondence of $\cN_{i,j}$ with $\cQ_{i,j}$) from the one in~\cite{FP17}. 
Indeed, in their bijection, the 
parameter $\alpha_0$ corresponds to the length (minus $1$)
of the leftmost branch in their labelled DFS trees. But that parameter does not correspond to a face-degree (e.g. for one of the two faces adjacent to the root-edge) nor to a vertex-degree (e.g. for one of the two extremities of the root-edge)  in the associated  non-separable map. 
\end{remark}

Remark~\ref{rk:further} also yields a model of maps for $m$-Tamari intervals. 
 For $m\geq 1$, we let $\cQ_n^{(m)}$ be the subfamily of $\cQ_{mn,n}$ where each inner white vertex has $m$ incoming blue edges in the minimal separating decomposition, and $s'$ has no incoming blue edge.  
Then $\Phi'$ yields a bijection between $\cQ_n^{(m)}$ and intervals of $\Tam_n^{(m)}$. 
It is known~\cite{BFP12} (extension of the formula for $m=1$ discovered in~\cite{Ch06}) that the number $I_n^{(m)}$ of intervals in $\Tam_n^{(m)}$ is given by the beautiful formula
\begin{equation}\label{eq:mtam}
I_n^{(m)}=\frac{m+1}{n(mn+1)}\binom{(m+1)^2n+m}{n-1}.
\end{equation}
The family $\cQ_{n}^{(1)}$ is in bijection (via contraction of the blue edges directed toward 
a white vertex~\cite[Section 5]{FPS09}) with simple triangulations with $n+3$ vertices, endowed with their minimal Schnyder wood. 
Under this correspondence, one can check that our bijection coincides with the one by Bernardi and Bonichon~\cite{BB09}  
(recalled and exploited in Section~\ref{sec:bij_sch}) between $\cI_n$ and simple triangulations with $n+3$ vertices.
Simple triangulations with $n+3$ vertices can then be bijectively enumerated (we will recall a correspondence
to certain mobiles in Section~\ref{sec:bij_sch}), giving a bijective proof of~\eqref{eq:mtam} for the case $m=1$.  
It would be interesting to provide a bijective proof of~\eqref{eq:mtam} working for all $m\geq 1$, based on such an approach (edge-contractions or similar operations applied to maps in $\cQ_n^{(m)}$, so as to obtain maps or hypermaps amenable to bijective enumeration). 

\medskip
\medskip

\begin{figure}[h!]
\begin{center}
\includegraphics[width=12cm]{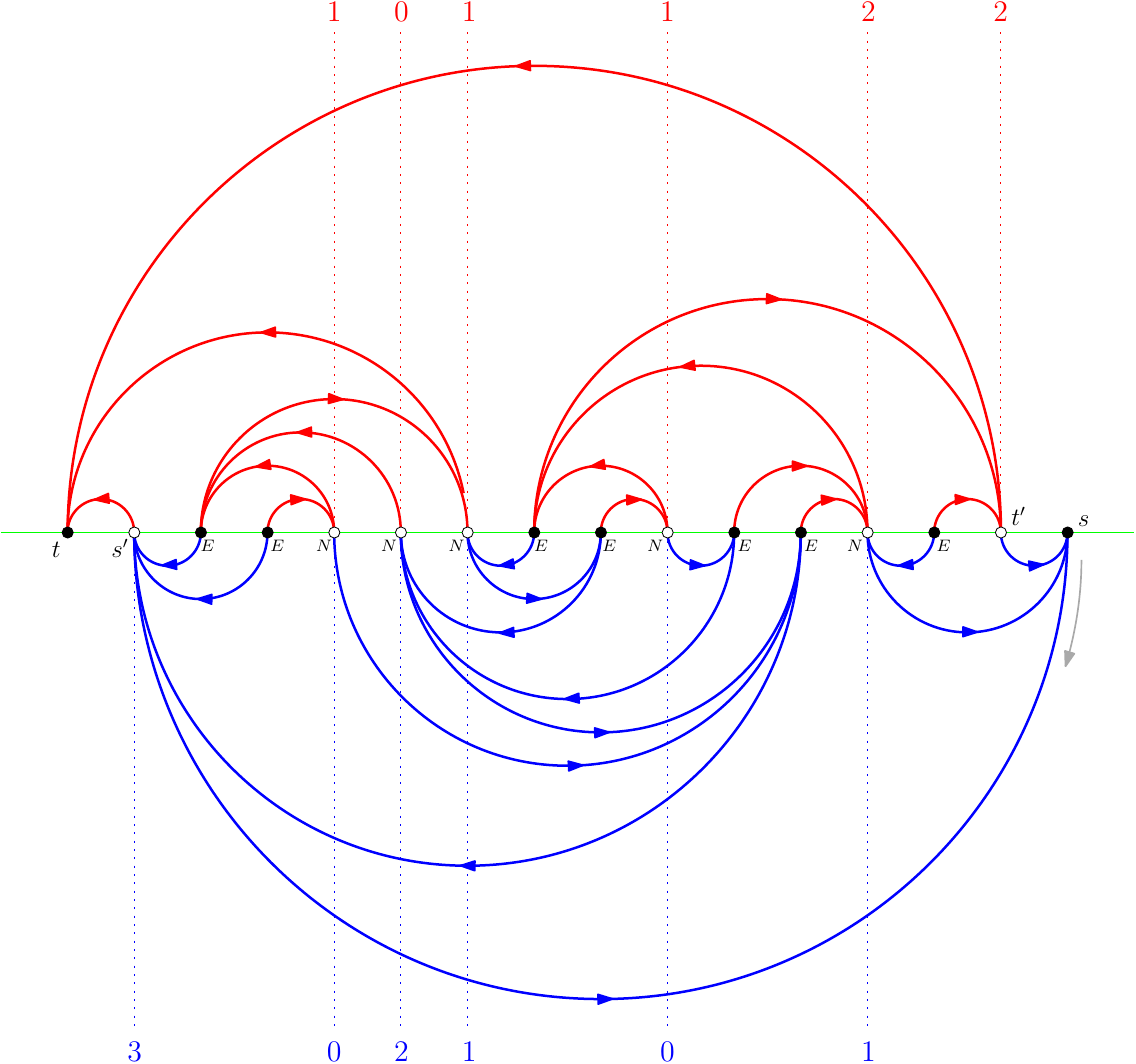}
\end{center}
\caption{The 2-book embedding of the separating decomposition of Figure~\ref{fig:separating}(b).
The first-visit order of white vertices around the blue tree corresponds to the left-to-right order along 
the line, and the middle word (when applying $\Phi'$) is obtained by reading the inner vertices left to right along the line, writing $E$ (resp. $N$) at every black (resp. white) vertex.} 
\label{fig:book_embedding}
\end{figure}

\subsection{A more symmetric formulation of the bijection $\Phi'$}\label{sec:symmetr}   
We reformulate here the bijection $\Phi'$ in a more symmetric way (in terms of the roles played by blue and red edges).  
We recall from~\cite{FFNO11, FKHO10} the notions of equatorial line and 2-book embedding of a separating decomposition $S=(Q,X)$.  
A separating decomposition has the property that each inner face $f$ has two bicolored corners. We may then draw a green curve inside $f$ to join these two corners.   
It is shown in Lemma 3.1 of~\cite{FFNO11} that the union of these green curves forms a simple curve $L$   
from $s'$ to $t'$ that visits all inner faces and all vertices except $s$ and $t$; then $L$ is called the \emph{equatorial line} of $S$,  see Figure~\ref{fig:separating}(c). 

One can then stretch $L$ into a horizontal line where the vertices of $Q$ are equally spaced (with $t$ as the left
extremity and $s$ as the right extremity), and planarly draw the blue (resp. 
red) edges of $S$ as half-circles in the lower (resp upper) half-plane, see Figure~\ref{fig:book_embedding}. 
This canonical drawing is called the 
\emph{2-book embedding} of $S$. It is shown in Theorem 2.14 of~\cite{FKHO10} (see also Proposition 3.3 in~\cite{FFNO11}) 
that the 2-book embedding of $S$ satisfies the following so-called \emph{alternating condition}: 

\medskip

{\bf (A)}: for all blue (resp. red) edges the right (resp. left) extremity is black. 

\medskip

This property, and planarity of the set of arcs (see the discussion on fingerprints in Section 3 of~\cite{FFNO11}), implies that when applying the bijection $\Phi'$, the word corresponding to the middle walk is exactly the word read by traversing the line from left to right (excluding the outer vertices $\{s,t,s',t'\}$), writing $E$ (resp. $N$) every time we meet a 
black (resp. white) vertex, see Figure~\ref{fig:book_embedding}. It also implies that the white vertices $v_0,\ldots,v_{j+1}$ (ordered according to first visits in a clockwise walk around the blue tree starting at the root-corner)
 are ordered left-to-right  along the line. 

For $S\in\Sep_{i,j}$ we let $\tau(S)$ be the separating decomposition obtained as the half-turn rotation of  $S$, i.e.,
the roles of $s$ and $t$ are exchanged and the colors are swapped. On the other hand, for 
$R=(\gammal,\gammam,\gammau)\in\cR_{i,j}$ we let $\tau(R)$ 
be the half-turn rotation of $R$, i.e., 
if $\mathrm{mir}(c_1,\ldots,c_n):=(c_n\ldots,c_1)$ denotes the mirror of a word on $\{N,E\}$,  
then $\tau(R):=(\mathrm{mir}(\gammau),\mathrm{mir}(\gammam),\mathrm{mir}(\gammal))$. 
Clearly, the symmetric reformulation of $\Phi'$ 
ensures that if $S\in\Sep_{i,j}$ is mapped by $\Phi'$ to $R\in\cR_{i,j}$, 
 then $\tau(S)$ is mapped to $\tau(R)$. Since $\tau$ is an involution on $\Sep_{i,j}$ that preserves the property
of being minimal, we obtain:

\begin{corollary}\label{coro:halfturn}
For $R\in\cR_{i,j}$ we have $\tau(R)\in\cG_{i,j}$ if and only if $R\in\cG_{i,j}$.
\end{corollary}
To be precise, in our proof of Theorem~\ref{theo:1} given next,    
we will use a characterization of elements of $\cG_{i,j}$
in terms of a certain arc diagram representation (Lemma~\ref{lem:Z_pattern}), such that 
 Corollary~\ref{coro:halfturn} already follows from this representation.

\subsection{Proof of Theorem~\ref{theo:1}}\label{sec:proof}
\subsubsection{Proof that $\Phi'$ is a bijection from $\Sep_{i,j}$ to $\cR_{i,j}$}

We define a \emph{blue-red arc diagram} as obtained by concatenating (for some $j\geq 0$) 
 $j+1$ horizontal segments $S_0,\ldots,S_j$, where for $r\in\llbracket 0,j\rrbracket$  
the segment $S_r$ is made of three successive (possibly empty) groups of dots: blue dots, then black dots, then red dots, such 
that the total numbers of blue dots, black dots, and red dots are the same;  
 in addition, in the upper (resp. lower) half-plane, there is a red (resp. blue) planar matching of the black dots with the red dots (resp. of the blue dots with the black dots), 
see the right-part of Figure~\ref{fig:arc} for an example. 
We denote by $\alpha_r$ (resp. $\mu_r,\beta_r$) the number of blue (resp. black, red) dots in $S_r$, for $r\in\llbracket 0,j\rrbracket$. 
We let $\cA_{i,j}$ be the set of blue-red arc diagrams made of $j+1$ segments
and having $i$ black dots.  
There is a straightforward bijection $\xi$ from $\cA_{i,j}$ to $\cR_{i,j}$:  the 
triple of walks $(\gammal,\gammam,\gammau)\in\cR_{i,j}$ associated to $A\in\cA_{i,j}$ is the one where $\gammal$ (resp. $\gammam$, $\gammau$)
has $\alpha_r$ (resp. $\mu_r$, $\beta_r$) East steps at height $r$ for $r\in\llbracket 0,j\rrbracket$. 
The property that $\gammam$ is weakly above $\gammal$ is equivalent to $\sum_{k=0}^r\alpha_k\geq \sum_{k=0}^r\mu_k$
for all $r\in\llbracket 0,j\rrbracket$, which is equivalent to the fact that the blue dots can be matched to the black
dots. Similarly, the property that $\gammau$ is weakly above $\gammam$ is equivalent to the fact that the black dots 
can be matched to the red dots. 

We will now describe a bijection $\chi$ from $\Sep_{i,j}$ to $\cA_{i,j}$. Before giving it, we state the following property that refines (A) and follows from~\cite[Sec.3]{FFNO11} (see
the discussion about uniqueness of alternating layouts of rooted plane trees before Proposition 3.3): 
\begin{lemma}\label{lem:2book}
Every separating decomposition has a unique 2-book embedding satisfying (A), the vertices being equally spaced on the equatorial line, with $t$ as the leftmost vertex and $s$
as the rightmost vertex. In this representation, for each vertex $v\notin\{s,t\}$, 
the outgoing edge (edge going to the parent) of $v$ in the upper (resp. lower) half-plane is the outermost arc incident to $v$.   
\end{lemma}

The bijection $\chi$ is done in two steps. Let $\Sep_{i,j}'$ be the set
of arc-diagrams specified as follows (see the left-part of Figure~\ref{fig:arc} for an example): 
\begin{itemize}
\item
There are $i+j+2$ vertices aligned along the horizontal axis, among which $i$ are black and $j+2$ are white, the left-most point $s'$ and rightmost point $t'$ being white.
\item
In the upper half-plane, there is a planar arc-system, such that the left (resp. right) extremity of every arc is black (resp. white), and every black vertex is incident to 
exactly one arc,
\item
In the lower half-plane, there is a planar arc-system, such that the left (resp. right) extremity of every arc is white (resp. black), and every black vertex is incident to 
exactly one arc.
\end{itemize} 

\begin{figure}
\begin{center}
\includegraphics[width=13.4cm]{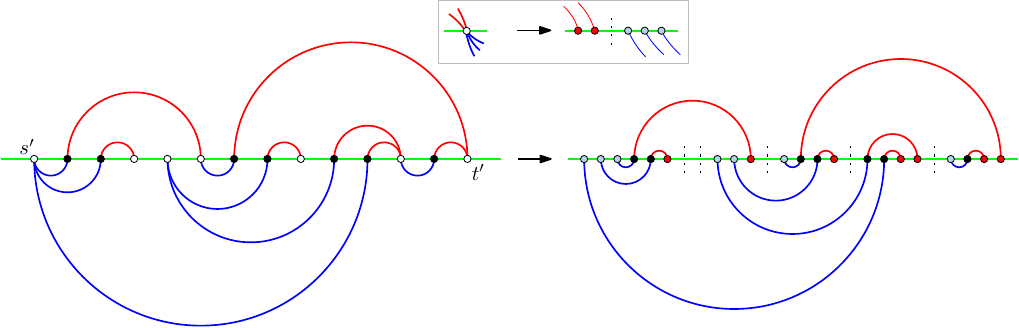}
\end{center}
\caption{Left: the image (in $\Sep_{7,5}'$) of the separating decomposition of Figure~\ref{fig:book_embedding} by the mapping $\chi_1$. Right: the corresponding blue-red arc diagram 
(in $\cA_{7,5}$) by the mapping $\chi_2$.}
\label{fig:arc}
\end{figure}

For $S\in\Sep_{i,j}$, let $S'$ be obtained as follows, see the left-part of Figure~\ref{fig:arc}:
\begin{itemize}
\item
take the $2$-book embedding of $S$ with property (A),
\item
delete the vertices $s,t$ and their incident edges,
\item
erase all edges with a white origin.
\end{itemize}
Let $\chi_1$ be the mapping that associates $S'$ to $S$.

\begin{lemma}
The mapping $\chi_1$ is a bijection from $\Sep_{i,j}$ to $\Sep_{i,j}'$.
\end{lemma}
\begin{proof}
For $S\in\Sep_{i,j}$, the fact that $\chi_1(S)\in\Sep_{i,j}'$ is a direct consequence of the property (A) of the 2-book embedding. 

The inverse construction is given as follows.    
For $S'\in\Sep_{i,j}'$, insert two black vertices at the left and right extremity on the line, called respectively $t$ and $s$. 
For $w$ a white vertex on the horizontal line, the \emph{upper parent} of $w$ is defined as follows: if 
$w$ is ``covered" by a red arc $a$ ($a$ is the first arc crossed by an upward vertical line starting from $w$), then the upper-parent of $w$ 
is the (black) vertex $b$ at the left extremity of $a$; otherwise the upper-parent of $w$ is $t$. Similarly, the \emph{lower parent} of $w$ is defined as follows: 
if $w$ is ``covered" by a blue arc $a$ ($a$ is the first arc crossed by a downward vertical line starting from $w$), then the lower parent of $w$ 
is the (black) vertex $b$ at the right extremity of $a$; otherwise the lower parent of $w$ is $s$. 
We let $S$ be obtained from $S'$ by orienting its arcs from black to white, then for each white vertex $w$, by inserting a red (resp. blue) arc from $w$ 
to its upper parent (resp. to its lower parent) in the upper (resp. lower) half-plane.  We claim that $S$ is a separating decomposition in $\Sep_{i,j}$, 
endowed with its unique $2$-book embedding as characterized in Lemma~\ref{lem:2book}. First, the arc insertions yield no crossing (e.g. 
for each upper arc $a$ of $S'$, letting $\Sigma_a$ be the face of the arc-system on the interior-side of $a$, all inserted arcs connected to the left extremity of $a$ occur within $\Sigma_a$).   
The alternation property (A) is clearly satisfied, as well as the fact that the outgoing edge of every vertex in the upper (resp. lower) half-plane is the outermost one. 
These two properties (and planarity of the arc-system) also easily imply that the graph in the upper (resp. lower) half-plane is a tree rooted at $t$ (resp. $s$), indeed
the directed path in the upper (resp. lower) half-plane starting from a given vertex gives a sequence of edges whose arcs must have increasing width, hence the path can not cycle, it thus has to end at the unique sink, which is $t$ (resp. $s$). Hence, the structure we obtain is indeed a separating decomposition $S$ endowed with its unique $2$-book embbedding satisfying
(A). Let $\delta_1$ be the mapping that sends $S'$ to $S$. Clearly, we have $\chi_1\circ\delta_1(S')=S'$ for every $S'\in\Sep_{i,j}'$. 
To check that $\delta_1\circ\chi_1(S)=S$ for every $S\in\Sep_{i,j}$, we observe that, for every white vertex $w\in S'=\chi_1(S)$, the unique black vertex allowed
to receive the outgoing arc of $w$ in the upper (resp. lower) half-plane is the upper parent (resp. lower parent) of $w$. Indeed, this is the only choice so as to satisfy planarity and 
the property that the outgoing edge at every black vertex $b$ is the outermost arc incident to $b$.  
\end{proof}

We now describe the second part (from $\Sep_{i,j}'$ to $\cA_{i,j}$), which is illustrated in Figure~\ref{fig:arc}.     
Let $S'\in\Sep_{i,j}'$, with $v_0,\ldots,v_{j+1}$ the white vertices from left to right along the horizontal axis. 
For $r\in\llbracket 0,j\rrbracket$, let $\alpha_r$ be the blue degree of $v_r$, and let $\beta_{r}$ be the red degree of $v_{r+1}$.  
We turn $v_0$ into a group of $\alpha_0$ blue dots, and $v_{j+1}$ into a group of $\beta_j$ red dots. 
Then, for $r\in\llbracket 1,j\rrbracket$, we turn $v_r$ into a group of $\beta_{r-1}$ red dots, followed by a segment-separator, followed by a group of $\alpha_{r}$ blue dots. 
We let $A$ be the obtained arc-diagram, and let $\chi_2$ be the mapping that associates $A$ to $S'$. 

\begin{lemma}
The mapping $\chi_2$ is a bijection from $\Sep_{i,j}'$ to $\cA_{i,j}$. 
\end{lemma}
\begin{proof}
Clearly, with the notation above, $A=\chi_2(S')$ is a blue-red arc-diagram, which has $j+1$ segments (initially, there is one segment, then each white vertex in $v_1,\ldots,v_j$ creates
a separator), with $\alpha_r$ (resp. $\beta_r$) blue (resp. red) dots in the $r$th segment for $r\in \llbracket 0,j\rrbracket$. 
The number $\mu_r$ of black dots in the $r$th segment is the number of black vertices between $v_r$ and 
$v_{r+1}$ on the horizontal line of $S'$ for $r\in \llbracket 0,j\rrbracket$. Since the number $i$ of black vertices becomes the number of black dots, $A$ is in $\cA_{i,j}$. The 
inverse mapping $\delta_2$ is defined so as to reverse the construction. For $A\in\cA_{i,j}$, we contract the group of blue dots in the first segment (with their incident blue arcs) into a white vertex $s'$ having blue degree $\alpha_0$ (and no incident red arc), we  contract the group of red dots in the last segment (with their incident red arcs) into a white vertex $t'$ having red degree $\beta_j$ (and no incident blue arc), and for each $r\in \llbracket 1,j\rrbracket$, we contract the group of $\beta_{r-1}$ red dots in the $(r-1)$th segment (with their
incident red arcs) and the group of $\alpha_{r}$ blue dots in the $r$th segment (with their incident blue arcs) into a white vertex $v_r$, which has blue degree $\alpha_{r}$ and
red degree $\beta_{r-1}$. We obtain an arc-diagram $S'\in\Sep_{i,j}'$ (the number of white vertices is $j+2$, and the number of black vertices is $i$, as it is equal to the number of 
black dots in $A$). Let $\delta_2$ be the mapping that associates $S'$ to $A$. By construction, the two mappings $\chi_2,\delta_2$ are inverse of each other. 
\end{proof}

To conclude, $\chi:=\chi_2\circ\chi_1$ is a bijection from $\Sep_{i,j}$ to $\cA_{i,j}$, and by construction, $\xi\circ\chi$ coincides with the symmetric reformulation of $\Phi'$ as 
given in Section~\ref{sec:symmetr}. Hence,  $\Phi'$ is a bijection from $\Sep_{i,j}$ to $\cR_{i,j}$.

\subsubsection{Proof that $S\in\Sep_{i,j}$ is minimal if and only if $\Phi'(S)\in\cG_{i,j}$}\label{sec:proof2} 
In a blue-red arc diagram, a \emph{Z-pattern} is a pair made of a blue arc $a$ and a red arc $a'$
such that the blue extremity of $a$ is enclosed within $a'$ and the red extremity of $a'$ is enclosed within $a$,
see Figure~\ref{fig:rrbb}(a). 

\begin{figure}
\begin{center}
\includegraphics[width=12cm]{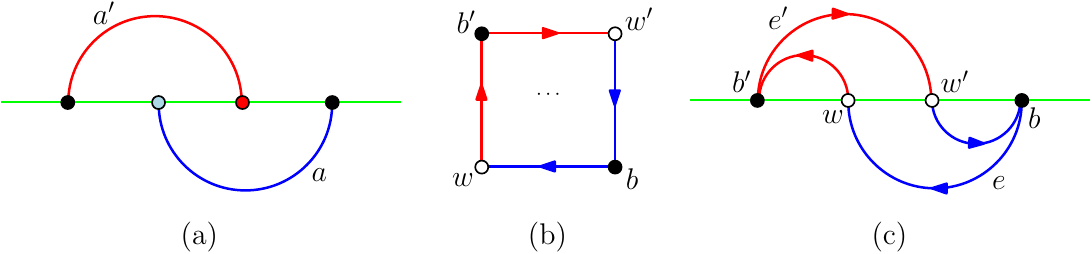}
\end{center}
\caption{(a) A Z-pattern in a blue-red arc diagram (drawing only the two involved arcs and their
extremities). (b) Situation for a  clockwise 4-cycle in a separating decomposition. (c) Situation for a clockwise 4-cycle in the 2-book embedding of a 
separating decomposition (drawing only the four involved vertices and edges).}
\label{fig:rrbb}
\end{figure}

Our proof that $S\in\Sep_{i,j}$ is minimal if and only if $\Phi'(S)\in\cG_{i,j}$ relies on the two following lemmas.

\begin{lemma}
Let $S\in\Sep_{i,j}$ and let $A\in\cA_{i,j}$ be the corresponding blue-red arc diagram (i.e., $A=\chi(S)$). 
Then $S$ is minimal
if and only if $A$ has no Z-pattern. 
\end{lemma} 
\begin{proof}
It is known (see e.g.~\cite[Prop.15]{bernardi2012schnyder}) that, if $S$ is not minimal, then it contains a clockwise 4-cycle $C$ (not necessarily the contour of a face),
and any edge in the interior of $C$ and incident to a vertex $v\in C$ is incoming at $v$. 
The local conditions (Figure~\ref{fig:separating}(a)) imply that the colors are as shown in Figure~\ref{fig:rrbb}(b). Then, in the 2-book embedding, the alternating
property (A) ensures that the $4$ edges of $C$ are as shown in Figure~\ref{fig:rrbb}(c). Hence, the two arcs of $A$ resulting from the two edges of $C$ going out of a black vertex form a Z-pattern. 

Conversely, assume $A$ has a Z-pattern. A Z-pattern is called \emph{minimum} if the distance along the line
between its blue dot and its red dot is smallest possible. 
Let $a,a'$ be a pair forming a minimum Z-pattern.   
Let $e=\{b,w\}$ and $e'=\{b',w'\}$ be the corresponding edges in the 2-book embedding of $S$. 
Let $e_1'$ be the outgoing red edge of $w$ and let $\tilde{b}$ be its black extremity. 
Assume $\tilde{b}\neq b'$. Then, by Lemma~\ref{lem:2book}, the outgoing red edge $e_2'$ of $\tilde{b}$ is in the area between $e_1'$ and $e'$, and 
it ends between $w$ and $w'$ ($w$ excluded, since the quadrangulation is simple). Let $a_2'$ be the arc of $A$ that corresponds to $e_2$. 
Then clearly the pair of arcs $a,a_2'$ forms a Z-pattern 
in $A$, contradicting the fact that the pair $a,a'$ is minimum. Hence $e_1'$ ends at $b'$. Similarly,  
the outgoing blue edge of $w'$ has to end at $b$. Hence the vertices $b,w,b',w'$ form a clockwise 4-cycle, which implies that 
$S$ is not minimal.  
\end{proof}

\begin{figure}
\begin{center}
\includegraphics[width=10.8cm]{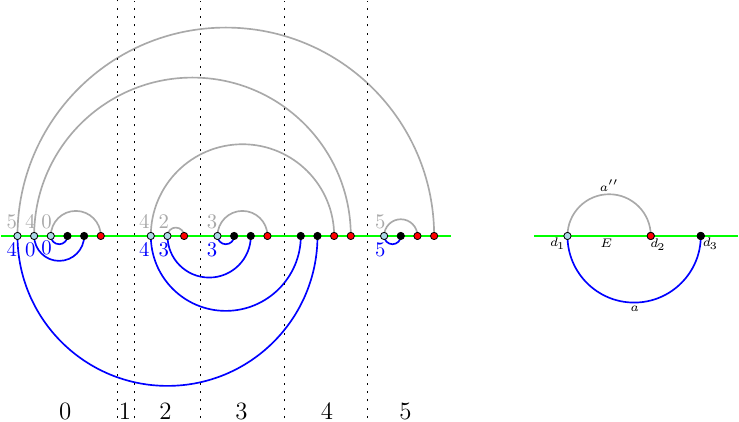}
\end{center}
\caption{Left: the modified arc diagram of the blue-red arc diagram shown in the right part of Figure~\ref{fig:arc}, 
which itself
corresponds to the triple $R=(\gammal,\gammam,\gammau)$ shown in Figure~\ref{fig:bijection}. The bracket vectors
of $\gammam$ and $\gammau$ with respect to $\nu:=\gammal$ are $\vec{U}=(4,0,0,4,3,3,5)$ and $\vec{V}=(5,4,0,4,2,3,5)$.
We have $U_5=3>2=V_5$, hence  $\gammam\not\leq\gammau$ in $\Tam(\nu)$. 
Right: the pattern to be avoided to have $\gammam\leq\gammau$ in $\Tam(\nu)$.}
\label{fig:modif_arc}
\end{figure}

\begin{lemma}\label{lem:Z_pattern}
Let $R=(\gammal,\gammam,\gammau)\in\cR_{i,j}$ and let $A\in\cA_{i,j}$ be the corresponding blue-red arc diagram (i.e.,
 $R=\xi(A)$). Then $R\in\cG_{i,j}$ 
if and only if $A$ has no Z-pattern. 
\end{lemma}
\begin{proof}

The \emph{modified arc diagram} $\tilde{A}$ of $A$ is the same as $A$ except that the arcs in the upper part 
match planarly the blue dots to the red dots instead of matching planarly the black dots to the red dots, see Figure~\ref{fig:modif_arc}.
 (the fact that the planar matching is doable in the upper half-plane follows from the fact that $\gammau$ is weakly above $\gammal$).   
Let $b_1,\ldots,b_i$ be the blue dots of $\tilde{A}$ ordered from left to right, and let $\nu:=\gammal$. 
For $k\in\llbracket 1,i\rrbracket$, let $U_k\in\llbracket 0,j\rrbracket$ be the index of the segment containing
the black dot matched with $b_k$, and let $V_k\in\llbracket 0,j\rrbracket$ be the index of the segment containing
the red dot matched with $b_k$, see the left part of Figure~\ref{fig:modif_arc}. The vectors   
$\vec{U}=(U_1,\ldots,U_i)$ and  $\vec{V}=(V_1,\ldots,V_i)$ are the \emph{bracket vectors}~\cite{ceballos2018nu} 
of $\gammam$ and $\gammau$
with respect to $\nu$ (compared to \cite{ceballos2018nu} we omit the fixed underlined entries).    
It is shown in Section~4 of~\cite{ceballos2018nu} that
 $\gammam\leq \gammau$ in $\Tam(\nu)$ iff $\vec{U}\leq \vec{V}$ (component by component). 
  
Clearly, this is equivalent to the fact
that the modified arc diagram $\tilde{A}$ avoids the pattern shown in the right part of Figure~\ref{fig:modif_arc}, i.e., for each blue point, its 
matched red point (via the incident arc in the upper half-plane) is on the right of its matched black point (via the incident arc in the lower half-plane).  
Indeed, the index of the segment to which the red dot belongs has to be greater or equal to the index of the segment to which the black dots belongs; 
since the group of red dots comes after the group of black dots in each segment, we conclude that the red dot has to be on the right of the black dot. 

Assume $R\notin\cG_{i,j}$. Then $\tilde{A}$ contains a pattern as in the right part of Figure~\ref{fig:modif_arc}. In this pattern, let $d_1,d_2,d_3$ be the blue dot, red dot and black dot, 
let $a''=(d_1,d_2)$ be the gray arc above and $a=(d_1,d_3)$ the blue arc below. Let $E$ be the 
part of the line strictly between $d_1$ and $d_2$, and   
let $\nblue,\nblack,\nred$ be respectively the numbers of blue dots, black dots,
and red dots in $E$. We have $\nblue=\nred$ since $d_1$ and $d_2$ are matched 
(for the upper diagram). We have $\nblue\geq\nblack$ since in the lower diagram $d_1$ is matched to a black dot 
that is on the right of $d_2$. Hence $\nred\geq\nblack$. In other word, in $E\cup d_2$ 
 we have more red dots than black dots. This implies that, in the upper diagram of $A$,   
there is a red dot in $E\cup d_2$ that is matched to a black dot on the left of $d_1$. 
Hence, if we let $a'$ be the arc formed by this matched pair, then the pair $a,a'$ is a Z-pattern in $A$. 

Assume now that $A$ contains a Z-pattern, and let $a,a'$ be a pair of arcs forming a minimum Z-pattern.  
We denote by $d_1$ the blue extremity of $a$, by $d_2$ the red extremity of $a'$, and by $E$ the 
part of the line strictly between $d_1$ and $d_2$. 
Let $\nblue,\nblack,\nred$ be respectively the numbers of blue dots, black dots,
and red dots in $E$.  
The fact that the pattern is minimum easily implies that there is no arc (neither in the upper nor in the lower
diagram) starting from $E$ and ending outside of $E$. Hence $\nblue=\nblack=\nred$. In the upper diagram of $\tilde{A}$, $d_1$ has thus to be matched with a red dot belonging to $E\cup d_2$. Clearly, this arc $a''$ together with $a$ form a pattern
as in the right part of Figure~\ref{fig:modif_arc}, hence $R\notin\cG_{i,j}$.    
\end{proof}

\medskip
\medskip

\begin{figure}
\begin{center}
\includegraphics[width=12cm]{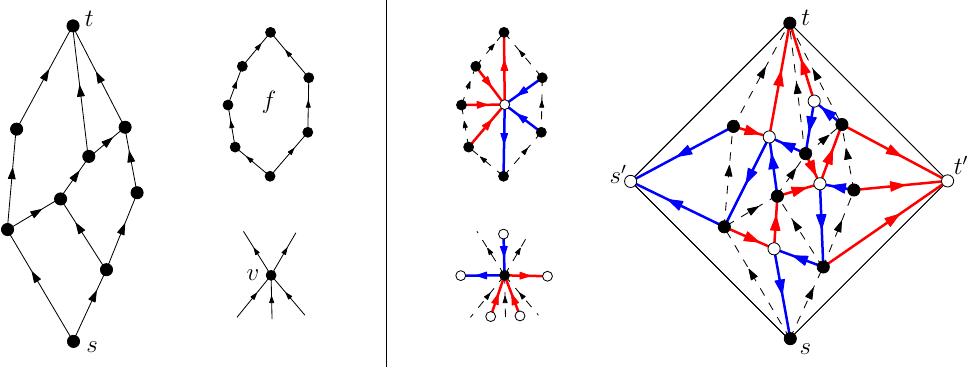}
\end{center}
\caption{Left: a plane bipolar orientation $B$, with the conditions at inner faces and at non-pole vertices. 
Right: the corresponding separating decomposition $S=\iota(B)$.}
\label{fig:bipolar}
\end{figure}

\section{Link to a bijection by Kenyon et al.~\cite{KMSW}} \label{sec:bijKMSW}

A \emph{plane bipolar orientation} is a map endowed
with an acyclic orientation having a unique source $s$ and a unique sink $t$ both incident to the outer face, with $s$ the root-vertex, 
see Figure~\ref{fig:bipolar} left.
It is known~\cite{FOP95} that a plane bipolar orientation is characterized by the following local conditions:
\begin{itemize}
\item
Every non-pole vertex (vertex $\notin\{s,t\}$)  has its incident edges partitioned into a non-empty interval of incoming edges
and a non-empty interval of outgoing edges. 
\item Every inner face has its incident edges partitioned
into a non-empty interval of clockwise edges and a non-empty interval of counterclockwise edges. 
\end{itemize}
An inner face consisting of $p+1$ clockwise edges and $q+1$ counterclockwise edges is said to have \emph{type $(p,q)$}.  
Let $\cB_{i,j}$ be the set of plane bipolar orientations with $i$ non-pole vertices and $j$ inner face,
and let $\cB_{i,j}[a,b]$ be the subset of those
 where the left (resp. right) outer boundary has length $a+1$ (resp. $b+1$). 
Let $\Sep_{i,j}[a,b]$ be the subset of $\Sep_{i,j}$ where $s'$ has degree $a+2$ and $t'$ has degree $b+2$. 
There is a direct bijection $\iota$ (illustrated in Figure~\ref{fig:bipolar}) from $\cB_{i,j}[a,b]$ to 
 $\Sep_{i,j}[a,b]$ where each vertex corresponds to a black vertex, and each inner face corresponds 
to an inner white vertex of the same type~\cite{FOP95}.

On the other hand, a \emph{tandem walk} is a 2d walk where each step is either $(1,-1)$ (SE step), or a step of the form $(-p,q)$ for some $p,q\geq 0$. 
Let $\cW_{i,j}[a,b]$ be the set of tandem walks of length $i+j$ with $i$ SE steps, starting
at $(0,a)$, ending at $(b,0)$ and staying in the quadrant $\mathbb{N}^2$. 
Let $\cR_{i,j}[a,b]$ be the subset 
of elements in $\cR_{i,j}$ where $\alpha_0=a$ and $\beta_j=b$. We now recall a bijection $\sigma$ from $\cR_{i,j}[a,b]$
to $\cW_{i,j}[a,b]$ recently described  in~\cite[Sec 9.1]{BFR19} (where we slightly change the convention: 
the lower and upper walks are constructed line by line here, and column by column in~\cite{BFR19}). For $R=(\gammal,\gammam,\gammau)\in\cR_{i,j}[a,b]$, we let $n=i+j$ and let $u_1,\ldots,u_n$ be the successive steps in $\gammam$. 
For $m\in\llbracket 1,n\rrbracket$, we let $s_m$ be a SE step if $u_m$
is a horizontal step, and let $s_m=(-p,q)$ if $u_m$ is the $r$th vertical step of $\gammam$ (with $r\in\llbracket 1,j\rrbracket$), where $p=\beta_{r-1}$ and $q=\alpha_{r}$. Then $\sigma(R)$ is defined as the walk (in $\cW_{i,j}[a,b]$) of length $n$ starting at $(0,a)$ and with successive steps $s_1,\ldots,s_n$, see Figure~\ref{fig:tandem} for an example. Combining the bijection $\Phi'$ with $\iota$ and $\sigma$, and with the parameter-correspondence of $\Phi'$ stated in Remark~\ref{rk:further}, we obtain:

\begin{prop}\label{prop:bij_tandem}
The mapping $\Lambda:=\sigma\circ\Phi'\circ\iota$ is a bijection from $\cB_{i,j}[a,b]$ to $\cW_{i,j}[a,b]$. Each non-pole 
vertex corresponds to a SE step, and each inner face of type $(p,q)$ corresponds to a step $(-p,q)$. 
\end{prop}

\begin{figure}
\begin{center}
\includegraphics[width=13.4cm]{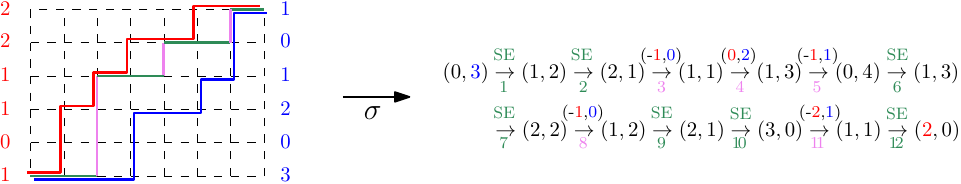}
\end{center}
\caption{Left: a triple of walks in 
$\cR_{7,5}[3,2]$ (the one obtained in Figure~\ref{fig:bijection}). Right: the corresponding (by $\sigma$)
tandem walk in $\cW_{7,5}[3,2]$.}
\label{fig:tandem}
\end{figure}

A bijection from $\cB_{i,j}[a,b]$ to $\cW_{i,j}[a,b]$ with the same parameter correspondence 
has been recently introduced by Kenyon et al.~\cite{KMSW}. 
We recall their consruction, and then prove that it coincides with $\Lambda$. 

For a bipolar orientation $B\in\cB_{i,j}[a,b]$, the \emph{rightmost tree}
 $T(B)$ of $B$ is the spanning tree of $B\backslash \{s\}$ obtained by selecting every edge
of $B$ that is the rightmost outgoing edge at its origin, with the exception that the rightmost outgoing edge $\tilde{e}$ of $s$ is not selected 
(see the left-part of Figure~\ref{fig:bij_KMSW}). 
An \emph{internal edge} is an edge in $T(B)$. An \emph{external edge} is an edge not in $T(B)$, and different from $\tilde{e}$.  
Clearly, there is a one-to-one correspondence between internal edges and non-pole vertices: every internal edge $e$ is the rightmost outgoing edge of exactly 
one non-pole vertex, which is denoted $v(e)$. 
Similarly, there is a one-to-one correspondence between external edges and inner faces: every external edge $e$ is the bottom-left edge of 
exactly one inner face, which is denoted $f(e)$.  

A counterclockwise 
walk around $T(B)$ yields an ordered list\footnote{More generally, such an ordering of the edges can be considered 
for any map endowed with a spanning tree, see~\cite{Be07} where
it is exploited to get bijective insights on the Tutte polynomial.} $\gamma=(e_1,\ldots,e_{i+j})$ of the edges of $B\backslash \tilde{e}$.  
Starting with $\gamma=\emptyset$, each time we walk along an internal edge $e$ away from the root $s$, we 
append $e$ to $\gamma$, and each time we cross the incoming half of an external edge $e$, we append $e$ to $\gamma$.
Let $W$ be the walk starting at $(0,a)$, with successive steps $s_1,\ldots,s_{i+j}$, 
obtained as follows. For every $k\in\llbracket 1,i+j\rrbracket$, if $e_k$ is internal, then $s_k$ is a SE step, while if 
$e_k$ is external, then $s_k=(-p,q)$, where $(p,q)$ is the type of the inner face $f(e_k)$, see Figure~\ref{fig:bij_KMSW} for an example. 
It is shown in~\cite{KMSW} that $W\in\cW_{i,j}[a,b]$, and that the mapping that associates $W$ to $B$ is a bijection from $\cB_{i,j}[a,b]$ to $\cW_{i,j}[a,b]$. 

\begin{remark}
We have presented here the KMSW with mirror conventions compared to~\cite{KMSW} (which relies on the leftmost tree of $B$). 
\end{remark}

\begin{figure}
\begin{center}
\includegraphics[width=12cm]{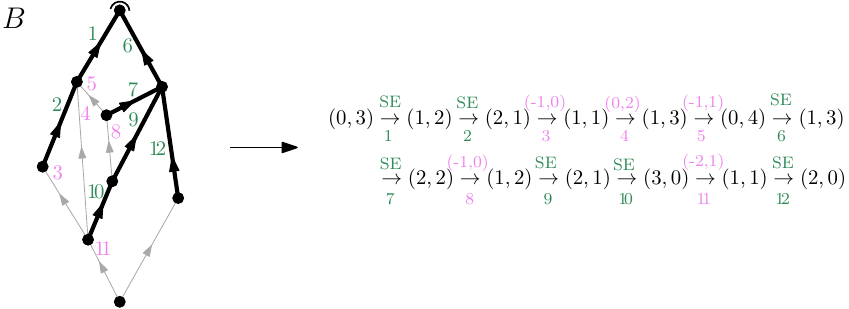}
\end{center}
\caption{A plane bipolar orientation $B\in\cB_{7,5}[3,2]$, and the corresponding tandem walk $W\in\cW_{7,5}[3,2]$ via the KMSW bijection.}
\label{fig:bij_KMSW}
\end{figure}

\begin{figure}
\begin{center}
\includegraphics[width=13.4cm]{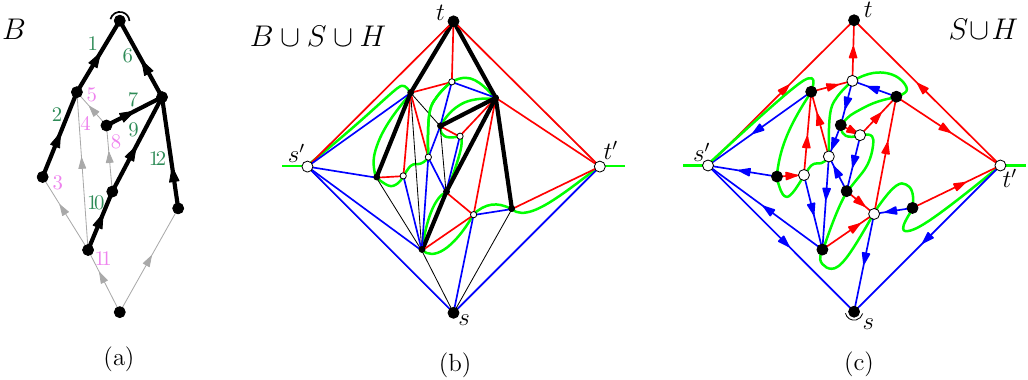}
\end{center}
\caption{Link between the constructions via the equatorial line.}
\label{fig:link}
\end{figure}

\begin{prop}
The bijection $\Lambda$ coincides with the KMSW bijection (with mirror conventions).
\end{prop}
\begin{proof}
Let $B$ be a plane bipolar orientation with $n+1$ edges, and left boundary of length $a+1$. Let $S=\iota(B)$ be the corresponding separating decomposition.  
The bijection $\Lambda$ amounts to visit the inner vertices of $S$ along the equatorial line $H$, producing a SE step when visiting a black vertex, and
producing a step $(-p,q)$ when visiting a white vertex having $p$ incoming red edges and $q$ incoming blue edges (the produced walk starting at $(0,a)$).  
Let $\tilde{e}$ be the bottom edge on the right outer boundary of $B$, and let $L=(e_1,\ldots,e_n)$ be the ordered list of edges of $B\backslash \tilde{e}$ used in the KMSW bijection.
For every edge $e$ of $B\backslash \tilde{e}$ that is internal (resp. external), we let $o(e)$ be the black (resp. white) inner vertex of $S$ corresponding to $v(e)$ (resp. to $f(e)$).  
Via the mapping $\iota$, the KMSW bijection amounts to produce (with starting point $(0,a)$) the step sequence $s_1,\ldots,s_n$, where, for $k\in \llbracket 1,n\rrbracket$, $s_k$ is a SE step if $o(e_k)$
is black, and is a step $(-p,q)$ if $o(e_k)$ is a white vertex having $p$ incoming red edges and $q$ incoming blue edges. 
Thus, we just have to check that $o(e_1),\ldots,o(e_n)$
gives the list of inner vertices of $S$ ordered along $H$.  
This property (which can be visualised in Figure~\ref{fig:link}(b)) amounts to check that, for two consecutive edges $e',e''$ along $L$, the vertices $o(e')$ and $o(e'')$ are
adjacent on $H$. This can be checked by a case-by-case analysis. Consider the case where $e'$ is external, and let $f(e')$ 
be the corresponding inner face of $B$, i.e., $e'$ is the bottom edge on the left boundary of $f(e')$. Note that, apart from $e'$,  all the edges on the left boundary of $f(e')$ are internal.    
 This easily implies that $e''$ (whether external or internal) has to be the top edge on the right boundary of $f(e')$.  An easy inspection ensures that, if $e''$ is internal (resp. external), then the white vertex $o(e')$ is adjacent to the black (resp. white) vertex $o(e'')$ on $H$. On the other hand, if $e'$ is internal, let $v(e')$ be the corresponding non-pole vertex, i.e., $v(e')$ is the origin of $e'$. Then $e''$ (whether external or internal) has to be the leftmost incoming edge at $v(e')$. Again, by inspection, if $e''$ is internal (resp. external),  
 then the black vertex $o(e')=v(e')$ is adjacent to the black (resp. white) vertex $o(e'')$ on~$H$.
\end{proof}

\medskip
\medskip 

\noindent {\bf Remark.} Another bijection from $\cB_{i,j}$ to $\cR_{i,j}$ is presented in~\cite[Sec 4.1]{AlPo}. 
It relies on the rightmost incoming tree of the (dual) bipolar orientation (i.e., the tree formed by the rightmost incoming edges of non-pole vertices),   
and it  is closely related to the bijection $\sigma^{-1}\circ\Lambda$. 
However, similarly as when taking $\Phi$ instead of $\Phi'$, one of the three walks (the upper one) differs.

\section{Bijection using Schnyder woods}\label{sec:bij_sch}
We first recall the definitions of Schnyder woods and Schnyder labelings~\cite{schnyder1990embedding}. 
For $T$ a simple triangulation, the outer 
vertices are called $u_B,u_G,u_R$ in clockwise order, with $u_B$ the one incident to the root-corner. 
A \emph{Schnyder wood} of $T$ is  an orientation and coloration
(in blue, green or red) of every inner edge  of $T$ such that all edges incident to $u_B,u_G,u_R$
are incoming of color blue (resp. green, red), and every inner vertex has outdegree $3$ and satisfies the local condition 
shown in Figure~\ref{fig:schnyder}(a).  
A Schnyder wood induces a coloring of the corners. For each corner $c$ at an inner vertex $v$, the \emph{edge opposite to $c$} is the second outgoing edge encountered after $c$
in clockwise order around $v$. Then each corner at an inner vertex inherits the color of its opposite edge, and each corner at an outer vertex $v$ receives the color of $v$. 
It can be checked that,  
around each inner face, there is one corner in each color and these occur as blue, green, red in clockwise order. 
It is known that the local conditions of Schnyder woods imply that the graph in every color is a tree spanning all the internal vertices (plus the outer vertex of the same color, which is the root-vertex of the tree).    
A Schnyder wood is called \emph{minimal} if it has no clockwise cycle. Any simple triangulation has a unique minimal Schnyder wood~\cite{Fe03}. 

\begin{figure}
\begin{center}
\includegraphics[width=13.4cm]{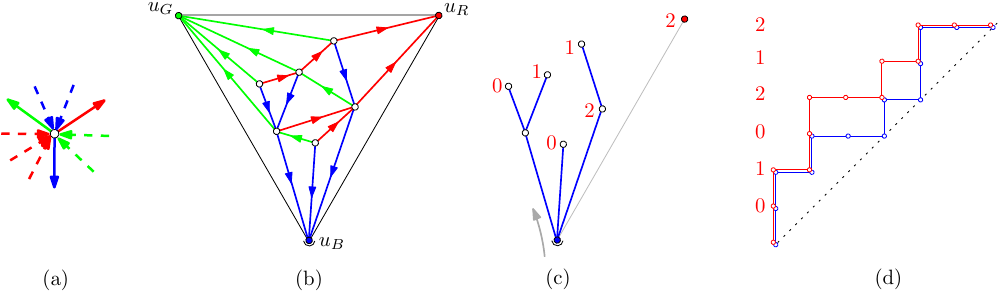}
\end{center}
\caption{(a) Local rule for inner vertices in Schnyder woods. (b) A Schnyder wood with $n+3$ vertices ($n=6$). 
(c) The blue tree with the indication of red indegrees at vertices.  
(d) The corresponding pair of Dyck walks (the red one above the blue one) in $\cP_n$.}
\label{fig:schnyder}
\end{figure}

Let $\cP_n$ be the set of pairs $(\gamma,\gamma')$ of Dyck walks of length $2n$ such that $\gamma'$ is weakly above $\gamma$. 
The Bernardi-Bonichon construction~\cite{BB09} 
starts from a simple triangulation with $n+3$ vertices endowed with a Schnyder wood, 
and outputs a pair $(\gamma,\gamma')\in\cP_n$. Precisely (see Figure~\ref{fig:schnyder} for an example), we let $\Tblue$ 
be the blue tree of the Schnyder wood plus the outer edge $e'=\{u_B,u_R\}$, and let $v_0,\ldots,v_{n}=u_R$ be the vertices
of $\Tblue\backslash\{u_B\}$ ordered according to the first visit in a clockwise walk around $\Tblue$ starting at $u_B$. 
Then $\gamma$ is obtained as the contour walk of $\Tblue\backslash e'$ and $\gamma'$ 
is $NE^{\beta_1}NE^{\beta_2}\cdots NE^{\beta_n}$, with $\beta_r$ the number of incoming red edges at $v_r$
 for $r\in\llbracket 1,n\rrbracket$. Bernardi and Bonichon 
show~\cite{BB09} that this construction gives a bijection between Schnyder woods on simple triangulations with $n+3$ vertices and 
$\cP_n$; and they show that it specializes into a bijection between minimal Schnyder woods with $n+3$ vertices 
 and $\cI_n\subset\cP_n$. Theorem~\ref{theo:1} can be seen as an extension of this statement to  separating decompositions, using
the fact~\cite[Section 5]{FPS09} that Schnyder woods correspond bijectively to separating decompositions where $s'$ has blue indegree $0$, 
and all inner white vertices have blue indegree $1$, and the bijection preserves the property of having no clockwise cycle.   

On the other hand, minimal Schnyder woods are themselves known to be in bijection to certain 
tree structures~\cite{BF12,PoSc06}. We will use here the bijection from~\cite{BF12}.  
A \emph{3-mobile} is a (non-rooted) plane tree $T$ where vertices have degree in $\{1,3\}$,  
respectively called \emph{leaves} and \emph{nodes}, and with an additional color structure given by the following conditions:
\begin{itemize}
\item 
The nodes are colored black or white, so that adjacent nodes have different colors, and there is at least one white node. 
\item
All leaves are adjacent to black nodes.
\item  
The edges are colored blue, green or red, such that around each node, the incident edges in clockwise order are blue, green and red.
\end{itemize}

In a 3-mobile, an edge is called a \emph{leg} if it is incident to a leaf and is called a \emph{plain edge}  otherwise. 
For $n\geq 1$, let $\cT_n$ be the set of 3-mobiles with $n$ white nodes. 
From a simple triangulation $M$ on $n+3$ vertices, endowed with its minimal Schnyder wood,    
one builds a 3-mobile $T\in\cT_n$ as follows (see Figure~\ref{fig:mobile}): 
\begin{enumerate}
\item
Orient the outer cycle clockwise.
\item
Insert a black vertex $b_f$ in each inner face $f$ of $M$.
\item 
For each edge $e=u\to v$, with $(f, f')$ the faces
on the left and on the right of $e$, create a plain edge $\{u,b_f\}$ (if $f$ is an inner face), and create  a leg at $b_{f'}$ pointing (but not reaching) to $v$.
Give to the plain edge (resp. to the leg) the color of the corresponding corner of $M$. 
\item
Erase the outer vertices and the edges of $M$. 
\end{enumerate}

Composing both constructions, we get a bijection between $\cI_n$ and $\cT_n$.  
Let $(\gamma,\gamma')\in\cI_n$, with $n\geq 1$. For $\binom{d}{c}\in\{\binom{E}{E},\binom{N}{N},\binom{E}{N}\}$, we say that a position 
$r\in\llbracket 0,n\rrbracket$ is of type $\binom{d}{c}$ if there is $c$ (resp. $d$) at position $r$ in $\can(\gamma)$ (resp. in $\can(\gamma')$). 
On the other hand, let $S$ be the minimal Schnyder wood with $n+3$ vertices associated to $(\gamma,\gamma')$ via the Bernardi-Bonichon construction, and let $T$ 
be the 3-mobile associated to $S$. As described above,  
we let $v_0,\ldots,v_{n}=u_R$ be the vertices
of $\Tblue\backslash\{u_B\}$ ordered according to the first visit in a clockwise walk around $\Tblue$ starting at $u_B$. 
For $r\in\llbracket 0,n\rrbracket$, let $e_r$ be the outgoing blue edge of $v_r$ (with the convention that $e_n$ is the outer edge $\{u_R,u_B\}$).  
Let $f_r$ be the face on the right of $e_r$, and let $b_r$ be the corresponding black node in  $T$. 
This gives a 1-to-1 correspondence between $\llbracket 0,n\rrbracket$ and black nodes of $T$ whose blue edge is a leg.  
Such a node is said to be of \emph{type $\binom{E}{E}$} if its green edge is a leg, of type \emph{$\binom{N}{N}$} if its red edge is a leg, 
and of \emph{type $\binom{E}{N}$} otherwise (only the blue edge is a leg), see the last column in Figure~\ref{fig:fin}.  
We claim that the type is preserved under this correspondence, i.e., a 
position $r\in\llbracket 0,n\rrbracket$ is of type $\binom{E}{E}$ (resp. $\binom{N}{N},\binom{E}{N}$) iff~$b_r$ is of the same type.  
Indeed, by the Bernardi-Bonichon construction, $\can(\gamma')$ has $E$ at position $r$ if and only if there is no incoming red edge at $v_r$, 
and $\can(\gamma)$ has $N$ at position $r$ if and only if $e$ comes just after a `valley' in a clockwise walk around $\Tblue$. 
From the local conditions of Schnyder woods, it is easy to see
that each of the 3 possible types for $r$ correspond to $f=f_r$ being in each of the configurations shown in the left column
of Figure~\ref{fig:fin}. By the local rules of the mobile constructions, these configurations correspond to $b=b_r$ being respectively of type 
 $\binom{E}{E}$, $\binom{N}{N}$, and $\binom{E}{N}$. To summarize, we obtain the following bijective result:

\begin{figure}
\begin{center}
\includegraphics[width=12.5cm]{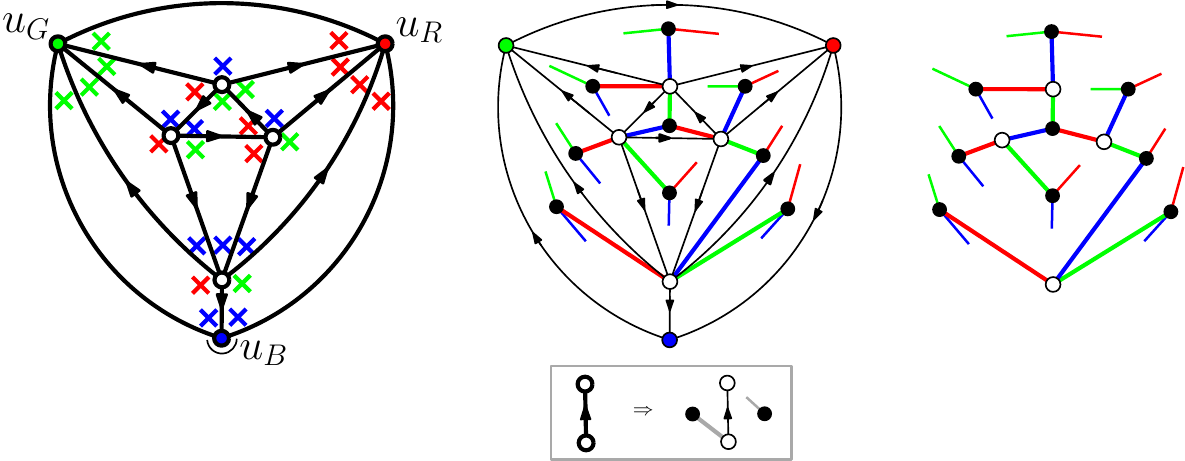}
\end{center}
\caption{Left: A simple triangulation endowed with its minimal Schnyder wood (colors are indicated at corners). Right:
the corresponding 3-mobile.}
\label{fig:mobile}
\end{figure}

\begin{figure}
\begin{center}
\includegraphics[width=10.8cm]{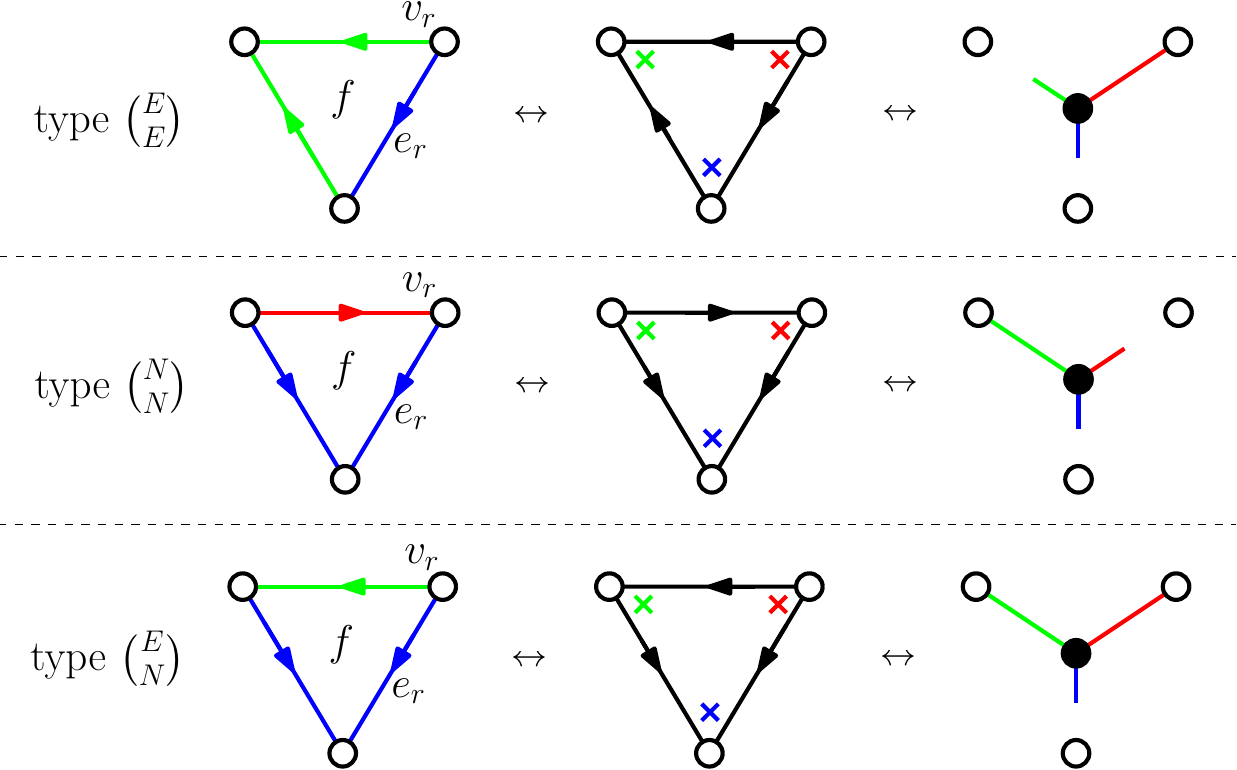}
\end{center}
\caption{Left column: configuration of $f$ for each of the possible types for position~$r$. Right column: configuration at the 
corresponding black vertex in the 3-mobile.}
\label{fig:fin}
\end{figure} 

\begin{theo}\label{theo:composed-bij}
Let $n\geq 1$. The composition of the Bernardi-Bonichon construction and of the 3-mobile construction gives a bijection between $\cI_n$ and $\cT_n$ such that each position $r\in\llbracket 0,n\rrbracket$ corresponds to a black node $b_r$ (whose blue edge is a leg) of the same type. 
\end{theo}

Let $i,j,k\geq 0$, and $n=i+j+k+1$. We denote by $a[i,j,k]$ the number of intervals in $\cI_n$ having $i+1$ positions of type $\binom{E}{E}$, 
$j+1$ positions of type $\binom{N}{N}$ and $k$ positions of type $\binom{E}{N}$, and we let $F(x,y,z):=\sum_{i,j,k}a[i,j,k]x^{i+1}y^{j+1}z^k$
be the associated generating function. Note that $F(x,y,0)=\sum_{i,j}|\cS_{i,j}|x^{i+1}y^{j+1}=\sum_{i,j}|\cG_{i,j}|x^{i+1}y^{j+1}$.  

\begin{corollary}\label{coro:tri_series}
The generating function $F\equiv F(x,y,z)$ is given by

\vspace{-.3cm}

\[F=xR+yG+zRG-\frac{RG}{(1+R)(1+G)},\] 


where $R,G$ are the trivariate series (in $x,y,z$) specified by the system 
\[
\left\{
\begin{array}{rl}
R&=(y+zR)(1+R)(1+G)^2,\\
G&=(x+zG)(1+G)(1+R)^2.
\end{array}
\right.
\]
\end{corollary}
Before proving the corollary, we note that   
$F(x,y,0)$ coincides (upon setting $G=u/(1-u)$ and $R=v/(1-v)$) with the known expression~\cite[Sec.2]{BT64}
 of the bivariate series $\sum_{i,j}|\cQ_{i,j}|x^{i+1}y^{j+1}$, and we recover 
$|\cG_{i,j}|=|\cQ_{i,j}|$ (we will also give a bijective argument at the end of the section); 
and $t+F(t,t,t)$ coincides (upon setting $G=R=\theta/(1-\theta)$) with the known expression~\cite[Eq.4.8-4.9]{Tu62} of the series counting  simple triangulations by the number of vertices minus~2. 
\begin{proof}
A \emph{planted 3-mobile} $T$ is defined similarly as a 3-mobile
 except that (exactly) one of the leaves is adjacent to a white node. 
This leaf is called the \emph{root} of $T$, and its incident edge is called the \emph{root-edge}. A planted 3-mobile
is called \emph{blue-rooted} (resp. \emph{red-rooted}, \emph{green-rooted}) if its root-edge is blue (resp. red, green). 
 We keep the same definition of 
black nodes of types $\binom{E}{E}$, $\binom{N}{N}$, $\binom{E}{N}$ as for 3-mobiles.  We let $B,R,G$ be the trivariate (variables $x,y,z$)  generating functions of blue-rooted, red-rooted, and green-rooted planted 3-mobiles, where $x$ (resp. $y$, $z$) is conjugate 
to the number of black nodes of type $\binom{E}{E}$ (resp. type $\binom{N}{N}$, type $\binom{E}{N}$). A 2-levels decomposition at the root 
translates into the following equation-system:
\[
\left\{
\begin{array}{rl}
B&=\big(B(1+G)+x+zG\big)\big(B(1+R)+y+zR\big),\\
R&=\big(B(1+R)+y+zR\big)(1+R)(1+G),\\
G&=\big(B(1+G)+x+zG\big)(1+R)(1+G).
\end{array}
\right.
\]
In a 3-mobile $T$, the number of blue legs is one more
than the number of blue plain edges. Indeed, letting $S$ be the associated minimal Schnyder wood,  there is a blue plain edge in~$T$ associated to every red edge (outgoing part) of $S$, and there is a blue leg associated to every blue edge (incoming part) of $S$, plus an extra blue leg associated to the outer edge $(u_B,u_R)$. 
Hence, by Theorem~\ref{theo:composed-bij}, 
$F=F_1-F_2$ where $F_1$ is the trivariate series of 3-mobiles with a marked blue leg, 
and $F_2$ is the trivariate series of 3-mobiles with a marked blue plain edge. A decomposition
at the marked blue leg gives $F_1=xR+yG+zRG$, and a decomposition at the marked blue plain edge gives $F_2=B(1+R)(1+G)$. We now simplify the equation-system 
by eliminating $B$. We look at the quantity $\frac{B(1+R)}{G}$, where we substitute $B$ and $G$ by their 
respective expressions in the equation-system. After simplification, this gives $\frac{B(1+R)}{G}=\frac{B(1+R)+y+zR}{1+G}$, 
so that $B(1+R)=yG+zRG$. Similarly, looking at the quantity $\frac{B(1+G)}{R}$, we obtain $B(1+G)=xR+zRG$. Substituting each occurence of $B(1+R)$ (resp. $B(1+G)$)
by $yG+zRG$ (resp. by $xR+zRG$) into the three-line system above, we obtain 
\[
\left\{
\begin{array}{rl}
B&=(1+R)(x+zG)(1+G)(y+zR),\\
R&=(y+zR)(1+R)(1+G)^2,\\
G&=(x+zG)(1+G)(1+R)^2.
\end{array}
\right.
\]
Then, substituting $B$ by its expression (given in the first line) into $F_2=B(1+R)(1+G)$, we obtain $F_2=(x+zG)(y+zR)(1+R)^2(1+G)^2=\frac{RG}{(1+R)(1+G)}$.
\end{proof}

\medskip
\medskip

\noindent{\bf A bijection between $\cG_{i,j}$ and $\cQ_{i,j}$ via mobiles.} A 3-mobile (with at least
one white node) is called \emph{synchronized} if it has no black node of type $\binom{E}{N}$, i.e., every black node  $b$ 
having a blue leg is incident to (exactly) one other leg. If this other leg is
 red (resp. green) then $b$ has type $\binom{N}{N}$
(resp. type $\binom{E}{E}$). We let $\cT_{i,j}^{\mathrm{syn}}$ be the set of synchronized 3-mobiles
with $i+1$ black nodes of type $\binom{E}{E}$ and $j+1$ black nodes of type $\binom{N}{N}$. It follows
from Theorem~\ref{theo:composed-bij} 
 that $\cT_{i,j}^{\mathrm{syn}}$ is in bijection with  $\cS_{i,j}$ (synchronized intervals such that the common canopy-word is in $\frak{S}(E^{i+1}N^{j+1})$), 
 itself in bijection with $\cG_{i,j}$. 

\begin{figure}
\begin{center}
\includegraphics[width=13.4cm]{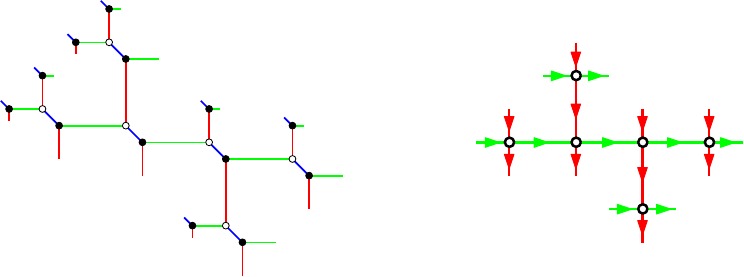}
\end{center}
\caption{Left: a synchronized 3-mobile in $\cT_{4,3}^{\mathrm{syn}}$. Right: the corresponding bicolored unrooted ternary 
tree in $\cU_{4,3}$.}
\label{fig:mobile_ternary}
\end{figure} 

An \emph{unrooted ternary tree} is a plane tree where all vertices have degree in $\{1,4\}$,
called respectively \emph{leaves} and \emph{nodes}. 
An unrooted ternary tree $T$ is said to be \emph{bicolored} if its edges are 
 colored green or red and are oriented such that, around each node, the incident edges in clockwise order are incoming red,
outgoing green, outgoing red, and incoming green. A leaf $\ell$ is called outgoing (resp. incoming) if 
its incident edge $e$ is outgoing (resp. incoming) at $\ell$, and is called red (resp. green) if $e$ is red (resp. green). 
It is easy to see that $T$ has as many red leaves that are outgoing as incoming, and as many green leaves that 
are outgoing as incoming. We let $\cU_{i,j}$ 
be the set of bicolored unrooted ternary trees having $i+1$ outgoing red leaves and $j+1$ outgoing green leaves. 
A bijection between $\cQ_{i,j}$ and $\cU_{i,j}$ has been 
introduced in Section 2.3.3 of~\cite{Sc98} and recovered in~\cite{BF12} (to obtain a ternary tree from a minimal separating decomposition, it actually uses the same local rules as those to obtain a 3-mobile from a minimal Schnyder wood, the inverse construction relies on the repeated use of so-called ``local closure" operations
 that yield the dual map of a simple quadrangulation).   

Hence, to derive another bijection between $\cG_{i,j}$ and $\cQ_{i,j}$ it remains to give a bijection
 between $\cT_{i,j}^{\mathrm{syn}}$ and $\cU_{i,j}$. The bijection, shown in Figure~\ref{fig:mobile_ternary}, is very simple. 
For $T\in\cT_{i,j}^{\mathrm{syn}}$, the corresponding $U\in\cU_{i,j}$ is obtained as follows: 
orient all the plain edges of $T$ from black to white nodes, then contract the blue plain edges, 
and finally delete the two legs at each 
black node of type $\binom{N}{N}$ or $\binom{E}{E}$ (the black nodes of type $\binom{E}{E}$ become outgoing red leaves,
those of type $\binom{N}{N}$ become outgoing green leaves).

\medskip
\medskip

\noindent{\bf Acknowledgement.} 
The authors are grateful to the two anonymous referees for very helpful comments and suggestions to improve the presentation, and 
 thank Fr\'ed\'eric Chapoton, Guillaume Chapuy, Wenjie Fang, and Mathias Lepoutre for interesting discussions.  
\'EF acknowledges the support of ANR-16-CE40-0009-01 ``GATO'', and AH the support of ERC-2016-STG 716083 ``CombiTop''.


\bibliographystyle{plain} 
\bibliography{biblio}

\end{document}